\documentclass[11pt,twoside]{article}

\usepackage{mathrsfs,amsfonts,amsmath,amssymb}
\usepackage{latexsym}
\newtheorem{satz}{Theorem}
\newtheorem{proposition}[satz]{Proposition}
\newtheorem{theorem}[satz]{Theorem}
\newtheorem{lemma}[satz]{Lemma}

\newtheorem{corollary}[satz]{Corollary}
\newtheorem{remark}[satz]{Remark}

\def\sbeq{\subseteq}

\def\Z{\mathbb {Z}}
\def\F{\mathbb {F}}
\def\Fp{\mathbb {F}_p}

\def\a{\alpha}

\def\d{\delta}
\def\g{\gamma}
\def\G{\Gamma}

\def\b{\beta}

\def\le{\leqslant}
\def\ge{\geqslant}
\def\_phi{\varphi}
\def\eps{\varepsilon}

\def\Gr{{\mathbf G}}
\def\FF{\widehat}

\def\ov{\overline}
\def\Spec{{\rm Spec\,}}

\def\t{\theta}

\def\D{\Delta}

\def\sf{\mathsf{sf}}

\begin{document}

\author{Tomasz Schoen\footnote{The author is supported by National Science Centre, Poland grant 2019/35/B/ST1/00264}, Ilya D. Shkredov
}
\title{ $L$--functions and sum--free sets
}
\date{}
	\maketitle

\begin{abstract}
	For  set $A\subset \F_p^*$ define by $\sf(A)$ the size of the largest sum--free subset of $A.$ Alon and Kleitman  \cite{AK}
	showed 
	that  $\sf (A) \ge |A|/3+O(|A|/p).$ We prove that if 
	$\sf(A)-|A|/3$ is small  then the set $A$ must be  
	uniformly distributed on cosets of each  large multiplicative subgroup. 
	Our argument relies on  irregularity of distribution of multiplicative   subgroups on certain   intervals in $\Fp$.  
\end{abstract}

\centerline{\it  Dedicated to the 80th birthday of Endre Szemer\'edi }

\section{Introduction}

Let $\Gr$ be  an additive group. For a finite set $A\sbeq \Gr \setminus\{e_\Gr\}$ denote by $\sf_k^{\Gr} (A)$ the maximal cardinality of a subset of $A$  without any solution to the equation 
\begin{eqnarray}\label{def:sf_k}
	x_1 + \dots + x_k = y \,.
\end{eqnarray}
If $k=2$, then sets without solutions to (\ref{def:sf_k}) are called {\it sum--free}  and hence 
 $\sf^\Gr_2 (A)$ is just the maximal size of a sum--free subset of $A$. 
Erd\H{o}s proved \cite{Erdos} that for every finite set $A\sbeq \Z\setminus\{0\}$ we always have 
$$\sf^\Z_2 (A) \ge |A|/3\,.$$ 
This estimate was slightly improved by Alon and Kleitman  \cite{AK}:  $\sf^\Z_2 (A) \ge (|A|+1)/3.$ In fact, they showed that $\sf^{\F_p}_2 (A) \ge (|A|+1)/3$ for any 
$A\sbeq \F_p^*$ provided that $p$ is a prime of the form $3k+2.$ The best known result was obtained by \cite{Bourgain} by Bourgain, who proved
  $$\sf^\Z_2(A) \ge (|A|+2)/3$$
for $n\ge 3$ using sophisticated Fourier analytical argument.
A sequence of results (see for example \cite{AK}, \cite{Lewko}, \cite{Alon}) provides upper estimates on $\sf^\Z_2(A).$ However, a breakthrough theorem was obtained by
Eberhard, Green and Manners \cite{EGM} who showed that  the constant $1/3$ is optimal by using a very elaborate technique. In a subsequent paper Eberhard 
\cite{eberhard} gave a simpler proof of this result that also holds for every $k\ge 2$ i.e. the optimal constant for $k-$sum-free sets is $1/(k+1).$ Although it is still unknown whether  
$\sf^\Z_2 (A) = |A|/3 + \psi (|A|)$ for $\psi (|A|) \to +\infty$ as $|A| \to +\infty$  it is widely believed to be true. 
Let us also mention that Bourgain \cite{Bourgain} showed  
$$
	\sf^\Z_3 (A) \ge \frac{|A|}{4}+ \frac{c \log |A|}{\log \log |A|} \,, 
$$
where $c>0$ is an absolute constant.

Our main purpose  is to  study the  quantity $\sf^{\Fp}_2 (A)$ for $A\subseteq \F_p^*$ which we  denote by  $\sf (A)$. 
The  main result concerns distribution of sets with small value of  $\sf (A)-|A|/3$. We show that such sets are equally distributed among cosets of multiplicative subgroups of $\Fp^*.$ 
We formulate below a  
slightly weaker theorem, which is directly implied by Theorem \ref{t:disc_sf} and Theorem \ref{t:disc_sf_G}. 

\begin{theorem}
	Let $p>3$ be a prime number and  let $\G \subseteq \F_p^*$ be a multiplicative subgroup with
$-1\in \G$. Let $A$ be a subset of $\F_p$  and suppose that $\sf(A) = |A|/3 + \psi$. 
Then for any $\eps>0$  one has 
\begin{equation}\label{f:disc_sf_G_intr}
	\sum_{\xi \in \F_p^*/\G} \left| |A \cap \xi \G| - \frac{|A||\G|}{p-1} \right|^2  \ll_\eps  \psi^2 p^{1+\eps} \,.
\end{equation} 
\label{t:main_intr}
\end{theorem}

It immediately  follows from Theorem \ref{t:main_intr} that if $A$ is a subset of quadratic residues then it contains a sum--free subset of size $|A|/3 + \Omega_\eps (|A| p^{-1/2-\eps})$, which improves   Bourgain's bound for $|A| \gg_\eps p^{1/2+2\eps}$. 


\bigskip 

\noindent {\it Acknowledgment.} We would like to thank Alexander Kalmynin for very useful  conversation about distribution of quadratic residues and Mateusz for linguistic  correction of the manuscript.

\section{Definitions and preliminaries}

Let $p$ be an odd  prime number.	
We denote by 
$\mathcal{Q}, \mathcal{N} \subseteq \F^*_p$  the set of quadratic residues and
 quadratic non--residues, respectively.
Let $\rho(x):=\left(\frac{x}{p}\right)$ be the Legendre symbol of $x \in \F_p$.
We denote by $\chi_0$ the principal (trivial) multiplicative character modulo $p$. 
Given a non--trivial character $\chi$, we write 
\[
	G(\chi) = \sum_{x \in \F_p} \chi(x) e^{2\pi i x/p} 
\]
for the corresponding {\it Gauss sum} and let 
\[
	L(s,\chi) = \sum_{m=1}^{\infty} \frac{\chi(m)}{m^s} 
\] 
be the {\it Dirichlet series} at point $s\in \mathbb{C}$.
By the famous result of Siegel for real multiplicative characters $\chi$, see, e.g. \cite{Montgomery-Vaughan} or \cite{Goldfeld_short}, 
we know that $L(1,\chi) \ge C_\eps p^{-\eps}$, where $C_\eps>0$ is an inefective constant. 
The best known effective lower bound  is $L(1,\chi) \gg \log p/\sqrt{p}$, see \cite{Goldfeld_1976}, \cite{GZ} and an excellent  survey \cite{Goldfeld_1985}.
Using partial summation method one can show a general inequality $|L(1,\chi)| \ll \log p$ which holds for any non--trivial character $\chi$. 
Furthermore, it is well--known that $|G(\chi)| = \sqrt{p}$ and 
$$
	\sum_{x \in \F_p} \chi(x) e^{2\pi i m x/p} = \ov{\chi}(m) G(\chi) \,. 
$$
	Using the Fourier expansion it is not hard to obtain a formula for sums with a character $\chi$ over the interval $[1,\a p]$, where  $\a \in (0,1)$, namely,  
\begin{equation}\label{f:S_a_expansion_gen}
S (\a) := \sum_{1 \le x \le \a p} \chi(x) = \frac{G(\chi)}{2\pi i} \sum_{m \neq 0} \frac{\ov{\chi} (m)}{m}  (1-e^{-2\pi i m \a} )  \,. 
\end{equation}

Let $\G \subseteq \F_p^*$ be a multiplicative subgroup and put $n=(p-1)/|\G|$.
Define by $\mathcal{X}$  the set of all  characters $\chi$ such that $\chi^n = \chi_0$, 
and put $\mathcal{X}^* = \mathcal{X} \setminus \{\chi_0\}$.
Note that  a character  $\chi$ belongs to $\mathcal{X}$ if and only if $\chi$ equals one on $\G$.  
For  interval $I= [p/3, 2p/3) \subset \F_p$ and  $\xi \in \F_p^*/\G$ put 
\[
\D_\xi = \D_\xi (\G)  = |\xi \G \cap I| - \frac{|I|}{n} \,.
\]
Clearly,  $\sum_\xi \D_\xi = 0$.
From now on
we do not 
underline 
that the summation over $\xi$ is taken over $\xi \in \F_p^*/\G$.  
Since $[p/3,2p/3) = -[p/3,2p/3)$ modulo $p$ it follows that $\D_{-\xi} = \D_{\xi}$. 

The notation $A(\cdot)$ always means the characteristic function of   set $A\subseteq \F_p.$ 
We denote the Fourier transform of  a function  $f : \Fp \to \mathbb{C}$ by~$\FF{f},$ 
\begin{equation}\label{F:Fourier}
\FF{f}(\xi) =  \sum_{x \in \Fp} f(x) e^{-2\pi i \xi x/p}\,.
\end{equation}
The Parseval formula states that
\begin{equation}\label{F_Par}
\sum_{x\in \Fp} |f(x)|^2
=
\frac{1}{p} \sum_{\xi \in \Fp} |\widehat{f} (\xi)|^2 \,.
\end{equation}
All logarithms are to base $2$ and the signs $\ll, \gg$ are the usual Vinogradov symbols.

\section{The case of quadratic residues}

We begin  with a result on  irregularity of the distribution of  quadratic  residues on the largest sum-free subset of $\F_p$.

\begin{lemma}
	Let $p$ be a prime number, $p>3$ and $p \equiv 1 \pmod 4$.  
	Then for any $\eps \in (0,1/2)$ there is a constant  $C_\eps >0$ such that 
\begin{equation}\label{f:discrepancy_Q}
	|\mathcal{Q} \cap [p/3, 2p/3)| - \frac{p}{6} \le - C_\eps p^{1/2-\eps} \,.
\end{equation}
	Furthermore, if $p\equiv 3 \pmod 4$ then
	\begin{equation}\label{f:discrepancy_Q-1}
	|\mathcal{Q} \cap [p/8, 3p/8)| - \frac{p}{8} \le - C_\eps p^{1/2-\eps} \,.
\end{equation}
\label{l:discrepancy_Q-1}
\end{lemma}
\begin{proof} 
	Put $\sigma:=|\mathcal{Q} \cap [p/3, 2p/3)| - \frac12(\lfloor 2p/3\rfloor-\lfloor p/3\rfloor).$ 
	Clearly, by
	$$|\mathcal{Q} \cap [p/3, 2p/3)|-|\mathcal{N} \cap [p/3, 2p/3)|=\sum_{p/3 \le x\le 2p/3} \left(\frac{x}{p}\right)$$
	and 
	$$|\mathcal{Q} \cap [p/3, 2p/3)|+|\mathcal{N} \cap [p/3, 2p/3)|=\lfloor 2p/3\rfloor-\lfloor p/3\rfloor$$
	we have
	$\sigma = \frac12 \sum_{p/3 \le x\le 2p/3} \left(\frac{x}{p}\right)$. 
	Note that since  $p \equiv 1 \pmod 4$ we have  $\mathcal{Q} = - \mathcal{Q}$, therefore  it follows that $\sum_{x\in \F_p} \left(\frac{x}{p}\right) = 0$ and
	$$\sigma = -\sum_{1\le x <p/3} \left(\frac{x}{p}\right).$$ 
	It is well--known \cite[Theorem 7.4(i)]{Wright.S} that 
\begin{equation}\label{f:Legendre_p/3}
	\sum_{1\le x <p/3} \left(\frac{x}{p}\right) = 
	\frac{\sqrt{3p}}{2\pi}  L(1, \rho \chi_3) \,,
\end{equation}
	where $L(1, \rho \chi_3)$ is the value of the Dirichlet series
	with product of the Legendre symbols $\rho$ and $\chi_3$  at $s=1$,
	where 
	$\chi_3$ is the   non--principal character modulo $3$
	(in other words $\chi_3$ is an integer function with period three such that $\chi_3 (1) = 1$, $\chi_3 (-1) = -1$ and $\chi_3 (0)=0$).   
	By Siegel's theorem  for real multiplicative characters $\chi$,  we know that $L(1,\chi) \gg_\eps p^{-\eps}$, which completes the proof of (\ref{f:discrepancy_Q}). 

	Next, assume that $p\equiv 3 \pmod 4$.
	Let $\chi_8$ be the  non--principal real character modulo $8$
	(in other words $\chi_8$ is an integer function with period eight such that $\chi_8 (1) = \chi_8 (-1) = 1$, $\chi_8 (3) = \chi_8 (-3) = -1$ and $\chi_8$ vanishes otherwise)
	 and let $\d_4$ be a function such that $\d_4 (n) = 1$ if and only if $n \equiv 4 \pmod 8$ and zero otherwise.  
	We have  
\begin{equation}\label{f:cos}
	\cos \frac{\pi n}{4} = \frac{\chi_8 (n)}{\sqrt{2}}  + \d_0 (n) - \d_4 (n) \,, \quad \quad 
	\cos \frac{3\pi n}{4} = - \frac{\chi_8 (n)}{\sqrt{2}} + \d_0 (n) - \d_4 (n)
\end{equation}
	By using the Fourier 
	expansion \eqref{f:S_a_expansion_gen} 
	of the characteristic function of  interval $[0,\a p]$ one can easily obtain 
\begin{equation}\label{f:S_a_expansion}
	S (\a) = \frac{\sqrt{p}}{\pi} \left( L(1, \rho) - \sum_{n\ge 1} \left(\frac{n}{p}\right) \frac{\cos 2\pi n \a}{n} \right) \,. 
\end{equation}
	By \eqref{f:cos} and \eqref{f:S_a_expansion}, we derive
\[
	S(3/8) - S(1/8) = \frac{\sqrt{2p}}{\pi} L(1, \rho \chi_8) \,.
\]
	Again, 
	to conclude the  proof it is enough to apply  Siegel's theorem.
 
$\hfill\Box$
\end{proof}

\bigskip

Now we are ready to prove our first 
main 
result. 
We make use of an observation that  if $\mathcal{Q}$ or $\mathcal{N}$ has a large sum--free set and since $\mathcal{Q}$ is a multiplicative subgroup then the same should be true for  {\it all} relatively large subsets of $\mathcal{Q}$ or $\mathcal{N}$.

\begin{theorem}
	Let $A\subseteq \F^*_p$, $p \equiv 1 \pmod 4$ and suppose that $\sf(A) = |A|/3 + \psi$. 
	Then for any $\eps>0$ one has 
	\begin{equation}\label{f:disc_sf}
	\big| \sum_{x\in A} \left( \frac{x}{p} \right) \big| \ll_\eps  \psi  p^{1/2+\eps}\, 
\end{equation}
	or equivalently $\left| |A\cap \mathcal{Q}| - |A\cap \mathcal{N}| \right| \ll_\eps  \psi  p^{1/2+\eps}.$
 Furthermore, if $p \equiv 3 \pmod 4$ and  $\sf_3 (A) = |A|/4+\psi$, then \eqref{f:disc_sf} holds as well.
	\label{t:disc_sf}
\end{theorem}
\begin{proof} 
	Let $A_\mathcal{Q} = A \cap \mathcal{Q}$, $A_\mathcal{N} = A \cap \mathcal{N}$ and put 
	 $I = [p/3,2p/3)$. Let us define $\D$ by 
	$$\D:=\frac12(\lfloor 2p/3\rfloor-\lfloor p/3\rfloor)-|\mathcal{Q} \cap I|=|\mathcal{N} \cap I|-\frac12(\lfloor 2p/3\rfloor-\lfloor p/3\rfloor).$$
	From the previous lemma we know that  $\D$ is positive and $\D \gg_\eps p^{1/2-\eps}$. 
	We have
\begin{eqnarray*}
	\sigma_1 &:=& |\mathcal{N}|^{-1} \sum_{x\in \mathcal{N}} |xA \cap I|   
	= 
		|\mathcal{N}|^{-1} \sum_{x\in \mathcal{N}} ( |xA_\mathcal{Q} \cap I| + |xA_\mathcal{N} \cap I|) \\
	&=& 
		\frac{|A_\mathcal{Q}|}{|\mathcal{N}|} |\mathcal{N} \cap I| +  \frac{|A_\mathcal{N}|}{|\mathcal{N}|} |\mathcal{Q} \cap I|=\frac{|A|}{3} + \D \frac{|A_\mathcal{Q}| - |A_\mathcal{N}|}{|\mathcal{N}|} \\
		&=&
		\frac{|A|}{3} + \frac{\D}{|\mathcal{N}|}\sum_{x\in A} \left( \frac{x}{p} \right) +O(1)\,.
\end{eqnarray*}
	Similarly, 
\begin{eqnarray*} 
\sigma_2 &:=& |\mathcal{Q}|^{-1} \sum_{x\in \mathcal{Q}} |xA \cap I|  
		= 
	|\mathcal{Q}|^{-1} \sum_{x\in \mathcal{Q}} ( |xA_\mathcal{Q} \cap I| + |xA_\mathcal{N} \cap I|) \\
	&=& 
\frac{|A_\mathcal{Q}|}{|\mathcal{Q}|} |\mathcal{Q} \cap I| +  \frac{|A_\mathcal{N}|}{|\mathcal{Q}|} |\mathcal{N} \cap I| 
=
\frac{|A|}{3} + \D \frac{|A_\mathcal{N}| - |A_\mathcal{Q}|}{|\mathcal{Q}|}\\
	&=&  \frac{|A|}{3} - \frac{\D}{|\mathcal{Q}|}\sum_{x\in A} \left( \frac{x}{p} \right)+O(1)\,.
\end{eqnarray*}
	However, $I$ is a sum--free set, hence  for any $x\in \F_p^*$, 
	$|xA \cap I|$ does not exceed $\sf(A)$. 
	Therefore, both $\sigma_1, \sigma_2$ are at most $\sf(A)$, which proves  \eqref{f:disc_sf}. 
	The second assertion can be shown in a similar manner  with the interval $J = [p/8, 3p/8)$, that has no solutions to $x_1+x_2+x_3=y$. 
	This completes the proof.
$\hfill\Box$
\end{proof}

\begin{corollary}
	Let  $p \equiv 1 \pmod 4$ and let $A\sbeq \mathcal{Q}$ or $A\sbeq \mathcal{N}.$  
	Then for every $\eps>0$ a positive constant $C_\eps$ exists such that
	\begin{equation}\label{f:sf-Q}
	\sf(A)\ge \frac{|A|}{3}+  C_\eps |A| p^{-1/2-\eps} \,.
\end{equation}
\label{c:sf}
\end{corollary}



\section{The case of arbitrary multiplicative subgroups}

The study of distribution of an arbitrary multiplicative subgroup  $\G$ on sum--free intervals must be handled in a different way.  It is difficult to determine the sign of the discrepancy of intersections of $\G$  the intervals $I$, $J$  as we did in Lemma  \ref{l:discrepancy_Q-1}.
Nevertheless, it is possible to find  $L_2$--norm of such discrepancy on cosets of $\G$. 
The intervals $I$ and $J$ suit our problem perfectly because we can deal with intervals that have rational number endpoints with small specific denominators, see discussion in \cite[Section 4]{Kalmynin}.

Let us recall that $n=(p-1)/|\Gamma|,$
$\mathcal{X}=\{\chi: \chi^n=\chi_0\},$
$\mathcal{X^*} = \mathcal{X} \setminus \{\chi_0\}$ and
$$\D_\xi = \D_\xi (\G)  = |\xi \G \cap I| - \frac{|I|}{n} \,.$$

\begin{proposition}
	Let $\G \subseteq \F_p^*$ be a multiplicative subgroup and let $\eps \in (0,1)$ be any number.   
	Then 
	\begin{equation}\label{f:discrepancy_G}
	p \log^2 p \gg \sum_{\xi} \D^2_\xi \gg_\eps \left(1 - \frac{2}{n} \right) p^{1-\eps}  \,.
	\end{equation}	
	Moreover, if $n \neq 2,4$ then
\begin{equation}\label{f:discrepancy_G_n}
	\sum_{\xi} \D^2_\xi \gg \frac{p}{\log^2 p}  \,.
\end{equation}
The above inequalities hold for the discrepancy on the interval $J$ provided that $-1\notin \G$. 
\label{p:discrepancy_G}	
\end{proposition}
\begin{proof} 
	We have 
	$$
	\G(x) = \frac1n \sum_{\chi \in \mathcal{X}} \chi(x) = \frac1n \sum_{\chi \in \mathcal{X}^*} \chi(x) + \frac{\chi_0(x)}{n} 
	\,,
	$$
	hence 
	\begin{equation}\label{tmp:17.09_1}
	\D_\xi = |\xi \G \cap I| - \frac{|I|}{n} = \frac1n \sum_{\chi \in \mathcal{X}^*} \ov{\chi}(\xi) \sum_{p/3 \le x \le 2p/3} \chi(x)   \,.
	\end{equation}
	Using the expansion \eqref{f:S_a_expansion_gen}, we obtain
\begin{eqnarray*}
	\sum_{p/3 \le x \le 2p/3} \chi(x) &=& S(2/3) - S(1/3) = \frac{G(\chi)}{2\pi i} \sum_{m \neq 0} \frac{\ov{\chi} (m)}{m} (e^{-2\pi i m/3} - e^{-4\pi i m/3}) \\
	&=&
	-\frac{\sqrt{3} G(\chi)}{2\pi} \sum_{m \neq 0} \frac{\ov{\chi} (m) \chi_3 (m)}{m} 
	=
	-\frac{\sqrt{3} G(\chi)}{2\pi} L(1,\ov{\chi} \chi_3) (1+\ov{\chi}(-1)) \,,
\end{eqnarray*}
hence by  \eqref{tmp:17.09_1}
\begin{equation}\label{f:D_xi}
	\D_\xi = -\frac{\sqrt{3}}{2\pi n} \sum_{\chi \in \mathcal{X}^*} (1+\ov{\chi}(-1)) G(\chi) L(1,\ov{\chi} \chi_3) \ov{\chi}(\xi) \,.
\end{equation}
	Thus, by the Parseval formula applied for the quotient group $\F_p^*/\G$ 
\begin{eqnarray}\label{e:L_2}
	\sum_{\xi} \D^2_\xi &=& \frac{3}{4 \pi^2 n} \sum_{\chi \in \mathcal{X}^*} |1+\ov{\chi}(-1)|^2 |G(\chi)|^2 |L(1,\ov{\chi} \chi_3)|^2 \nonumber  \\ 
 &=&
	\frac{3p}{4 \pi^2 n} \sum_{\chi \in \mathcal{X}^*} |1+\ov{\chi}(-1)|^2 |L(1,\ov{\chi} \chi_3)|^2 \,.
\end{eqnarray}
	By Siegel's Theorem, we have 
\begin{eqnarray*}
	\sum_{\xi} \D^2_\xi &\gg_\eps& \frac{p^{1-\eps}}{n} \sum_{\chi \in \mathcal{X}^*} |1+\ov{\chi}(-1)|^2 \\
	&=&
	\frac{p^{1-\eps}}{n}  (2n-2 + \sum_{\chi \in \mathcal{X}^*} (\ov{\chi}(-1) + \chi(-1))) \\
	&\ge& \left(2 - \frac{4}{n} \right) p^{1-\eps}\,, 
\end{eqnarray*}
where the last  inequality follows from equations $\sum_{\chi \in \mathcal{X}^*} \chi(-1) = -1$ if $-1\notin \G$ and $\sum_{\chi \in \mathcal{X}^*} \chi(-1) = n-1$ if $-1\in \G$.  
	
	The upper bound in  \eqref{f:discrepancy_G} follows from  (\ref{e:L_2})  and from a general inequality $|L(1,\chi)| \ll \log p$, which holds for every non--trivial character $\chi$. 
	
	Now let us assume that $n\not=2,4$. We only bound from below the  subsum of 
	 (\ref{e:L_2})  over non--quadratic characters.
	It was proven in  \cite[Theorems 11.4 and 11.11]{Montgomery-Vaughan} that $|L(1,\chi)| \gg \frac{1}{\log p}$ for any complex character  $\chi$, so  

\[
	\sum_{\xi} \D^2_\xi \gg \frac{p}{n\log^2 p} (2n-4-4) \gg \frac{p}{\log^2 p} \left(1 - \frac{4}{n} \right) \,.
\]
	If $n=3$ then we have $-1\in \G$ hence $\sum_{\chi \in \mathcal{X}^*} \chi(-1) = n-1 = 2$  and thus \eqref{f:discrepancy_G_n} is also satisfied.


	Next, we prove  the last part of our proposition. 
	Applying formula \eqref{f:S_a_expansion_gen} again (one can use identity \eqref{f:cos} as well), we obtain 
\begin{eqnarray}\label{f:ugly}
&&\sum_{p/8 \le x \le 3p/8} \chi(x) = S(3/8) - S(1/8) = \frac{G(\chi)}{2\pi i} \sum_{m \neq 0} \frac{\ov{\chi} (m)}{m} (e^{-\pi i m/4} - e^{-3\pi i m/4}) \nonumber \\
&=&
\frac{G(\chi)}{\sqrt{2} \pi i} \sum_{m \neq 0} \frac{\ov{\chi} (m) \chi_8 (m)}{m} 
	+ 
		\frac{G(\chi)}{\pi} \left( \sum_{m \equiv -2 \pmod 8} \frac{\ov{\chi} (m) }{m} - \sum_{m \equiv 2 \pmod 8} \frac{\ov{\chi} (m) }{m} \right)\nonumber \\
&=&
	\frac{G(\chi)}{\sqrt{2} \pi i} (1-\ov{\chi} (-1)) \sum_{m \ge 1} \frac{\ov{\chi} (m) \chi_8 (m)}{m}
- \frac{G(\chi)}{\pi}  (1+\ov{\chi}(-1)) \sum_{m \equiv 2 \pmod 8} \frac{\ov{\chi} (m) }{m} \nonumber \\
	&=&
	\frac{G(\chi)}{\sqrt{2} \pi i} (1-\ov{\chi} (-1)) \sum_{m \ge 1} \frac{\ov{\chi} (m) \chi_8 (m)}{m} 
	- \frac{G(\chi) \ov{\chi} (2)}{2\pi}  (1+\ov{\chi}(-1)) \sum_{m \equiv 1 \pmod 4} \frac{\ov{\chi} (m) }{m} \,.
\end{eqnarray}
Let $\mathcal{X}^*_o$ be the set of all odd characters from $\mathcal{X}^*$, then 
	from $-1\notin \G$ it follows that $|\mathcal{X}^*_o| = n/2$. 
	Hence 
	for all $\chi \in \mathcal{X}^*_o$, we have  
\begin{equation}\label{tmp:20.09_1}
	\sum_{p/8 \le x \le 3p/8} \chi(x) = \frac{G(\chi) \sqrt{2}}{\pi i}  \sum_{m \ge 1} \frac{\ov{\chi} (m) \chi_8 (m)}{m} 
	= \frac{G(\chi) \sqrt{2}}{\pi i}  L(1,\ov{\chi} \chi_8)\,.
\end{equation}
	Thus, by \eqref{tmp:20.09_1}  one obtains
\[
	\sum_{\xi} \left||\xi \G \cap J| - \frac{|J|}{n} \right|^2 \gg \frac{p}{n} \sum_{\chi \in \mathcal{X}^*_o} |L(1,\ov{\chi} \chi_8)|^2  
	=	
	\frac{p}{n} \left(|L(1,\rho \chi_8)|^2 + \sum_{\chi \in \mathcal{X}^*_o \setminus \{\rho\}} |L(1,\ov{\chi} \chi_8)|^2  \right) \gg_\eps {p^{1-\eps}}
\]
This completes the proof. 
	$\hfill\Box$
\end{proof}


\begin{remark}
	It was proven in \cite[page 366]{Montgomery-Vaughan} that for odd characters $\chi \in \mathcal{X}^*$ one has $|L(1,\chi)| \gg (\log p)^{-\cos (\pi/n)}$. 
	Thus, for small $n$ the estimate \eqref{f:discrepancy_G_n} can be further improved. 
	It is also well--known that under the Generalized Riemann Hypothesis a stronger inequality is satisfied  $$(\log \log p)^{-1} \ll |L(1,\chi)| \ll \log \log p.$$  
\end{remark}

For any non--trivial  character $\chi$ one has $|L(1,\chi)| \ll \log p$, therefore by \eqref{f:D_xi}  it follows  that $|\D_{\xi}| \ll \sqrt{p} \log p$. Thus, we can derive the following corollary from Proposition \ref{p:discrepancy_G}.

\begin{corollary}
	If $n \neq 2,4$,
	then there exist $\xi, \eta, \omega  \in \F_p^*/\G$ such that $\D_\xi \gg \sqrt{p}/(n\log^3 p), ~\D_\eta \ll -\sqrt{p}/(n\log^3 p)$ 
	and $|\D_\omega| \gg \sqrt{p}/(\sqrt{n} \log p)$. 
\label{c:pm_D}
\end{corollary}

Our next theorem provides  an analogous  estimate  to \eqref{f:disc_sf} for an arbitrary multiplicative subgroup.

\begin{theorem}
	Let $\G \subseteq \F_p^*$ be a multiplicative subgroup, let
	$A$ be a subset of $\F_p$  and suppose that $\sf(A) = |A|/3 + \psi$. 
	Then for any $\eps>0$ and for every even character $\eta \in \mathcal{X}^*$ one has 
	\begin{equation}\label{f:disc_sf_G}
	\big| \sum_{x\in A} \eta (x) \big| \ll_\eps  \psi  n^{1/2} p^{1/2+\eps} \,.
	\end{equation} \label{t:disc_sf_G}
\end{theorem}
\begin{proof} 
	Without losing  generality, we can assume that $0 \notin A$. 
	For  $\xi \in \F^*_p/\G$ put $A_\xi = \xi \G \cap A$ and 
	$$
		a_\xi = |A_\xi| - \frac{|A||\G|}{p-1} = |\xi \G \cap A| - \frac{|A|}{n} \,.
	$$
	Clearly,  $\sum_{\xi} a_\xi = 0$. 
	First, let us assume that $-1 \notin \G$,
	then 
	the quotient group $\F^*_p/\G$ can be written   as  $\F^*_p/\G = H \bigsqcup (-H)$ for  some set $H \subset \F^*_p/\G$ with $H\cap (-H)=\emptyset.$
	By our assumption $\eta (-1) = 1$, we deduce that
\begin{eqnarray*}
\sum_{x\in A} \eta (x) &=& \sum_{\xi} \sum_{x\in A_\xi} \eta  (x) = \sum_{\xi} \eta  (\xi) |A_\xi| = \sum_{\xi} \eta  (\xi) a_\xi \\
	&=& 
		\sum_{\xi\in H} \eta  (\xi) ( a_\xi  + \eta (-1) a_{-\xi} ) = \sum_{\xi\in H} \eta  (\xi) ( a_\xi  + a_{-\xi} ) \,.
\end{eqnarray*}
	and by the Cauchy--Schwarz inequality 
	\begin{equation}\label{f:chi_a_xi}
		|\sum_{x\in A} \eta (x) |^2 \le \frac12 n \cdot \sum_{\xi} (a_\xi + a_{-\xi})^2\,. 
	\end{equation}
	If $-1 \in \G$ then $a_\xi = a_{-\xi}$ and the inequality  \eqref{f:chi_a_xi} holds  as well.

	Again let $I = [p/3,2p/3)$. 
	For any $\a \in \F^*_p/\G$ in view of $\sum_{\xi} a_\xi =  \sum_{\xi} \D_{\xi} = 0$, we get  
\begin{eqnarray}\label{tmp:18.09_2}
	|\G|^{-1} \sum_{x \in \a \G} |xA \cap I| &=& |\G|^{-1}  \sum_\xi \sum_{x \in \a \G} |xA_\xi  \cap I| 
	= |\G|^{-1}  \sum_\xi |A_\xi| \left(\D_{\a \xi} + \frac{|I|}{n} \right) \\
	&=& 
	\frac{|A||I|}{p-1} +  |\G|^{-1}  \sum_\xi |A_\xi| \D_{\a \xi} = \frac{|A||I|}{p-1} + |\G|^{-1}  \sum_\xi a_\xi \D_{\a \xi} \,. \nonumber
\end{eqnarray}
	By $\sf(A) = |A|/3 + \psi$  the left--hand side of \eqref{tmp:18.09_2} does not exceed  $|A|/3+ \psi$, which implies
	that for any  $\a \in \F^*_p/\G$ the following inequality holds 
\begin{equation*}
	\sum_\xi a_\xi \D_{\a \xi} \le \psi |\G|+O(|A||\G|/p) \,.
\end{equation*}
	Since $\sum_{\xi} a_\xi =  \sum_{\xi} \D_{\xi} = 0$ it follows that $\sum_\a \sum_\xi a_\xi \D_{\a \xi}=0$ and 
\begin{equation*}
		\sum_\a |\sum_\xi a_\xi \D_{\a \xi}| \le 2\psi |\G| n +O(|A|/p)\le 2 \psi p +O(|A|) \ll \psi p 
\end{equation*}
	because $\psi \gg 1$. 
	Splitting 
	the set  $\F^*_p/\G = H \bigsqcup (-H)$ if it is possible and 
	using  
	the property $\D_{\xi} = \D_{-\xi}$ and 
	the last inequality we deduce that 
\begin{equation}\label{f:D_via_sf-}
\sum_{\a} |\sum_\xi a_\xi \D_{\a \xi}|^2 = \sum_{\a \in H} |\sum_\xi \D_{\a \xi} (a_\xi + a_{-\xi})|^2 
\ll 
	\psi p \cdot \max_{\a \in H} |\sum_\xi \D_{\a \xi} (a_\xi + a_{-\xi})| \,.
\end{equation}
	Again, if $-1\in \G$, then \eqref{f:D_via_sf-} also takes place. 
Therefore, by the
upper bound  \eqref{f:discrepancy_G}  and \eqref{f:D_via_sf-} we have
\begin{equation}\label{f:D_via_sf}
	\sum_{\a} |\sum_\xi a_\xi \D_{\a \xi}|^2 \ll \psi \left( \sum_{\xi} (a_\xi + a_{-\xi})^2 \right)^{1/2} \cdot p^{3/2} \log p \,.
\end{equation}
	On the other hand, from  \eqref{f:D_xi} and the Parseval identity, we obtain 
\begin{equation}\label{tmp:25.09_1}
	\sum_{\a} |\sum_\xi a_\xi \D_{\a \xi}|^2 
		= 
			\frac{3}{4\pi^2 n} \sum_{\chi \in \mathcal{X}^*} |1+\ov{\chi}(-1)|^2 |G(\chi)|^2 |L(1,\ov{\chi} \chi_3)|^2  |\sum_{\xi} a_\xi \ov{\chi} (\xi) |^2 \,.
\end{equation}
	Applying  Siegel's Theorem one more time  and using $\sum_{\xi} a_\xi =0$, we get
\begin{eqnarray}\label{f:E(A,D)_below}
	\sum_{\a} |\sum_\xi a_\xi \D_{\a \xi}|^2 &\gg_\eps& \frac{p^{1-\eps}}{n}  \sum_{\chi \in \mathcal{X}^*} |1+\ov{\chi}(-1)|^2 |\sum_{\xi} a_\xi \ov{\chi} (\xi) |^2 \nonumber\\
	&=&
	\frac{p^{1-\eps}}{n}  \sum_{\chi \in \mathcal{X}} |1+\ov{\chi}(-1)|^2 |\sum_{\xi} a_\xi \ov{\chi} (\xi) |^2  \,.
\end{eqnarray}
	Expanding  the last sum, we derive 
\begin{equation}\label{f:final_a_xi}
	\sum_{\a} |\sum_\xi a_\xi \D_{\a \xi}|^2 \gg_\eps p^{1-\eps} \sum_{\xi} (a_\xi + a_{-\xi})^2 \,.
\end{equation}
	Combining \eqref{f:D_via_sf} and  \eqref{f:final_a_xi} we have
\[
	\sum_{\xi} (a_\xi + a_{-\xi})^2 \ll_\eps \psi^2 p^{1+3\eps} \,.
\]
	Substituting the last formula in \eqref{f:chi_a_xi}, we obtain 
\[
	|\sum_{x\in A} \chi(x) |^2  \ll_\eps n \psi^2  p^{1+3\eps}
\]
	as required. 
$\hfill\Box$
\end{proof}

\bigskip




\noindent The last  result in this section provides an estimate of  $\sf(A)$ for sets with a very small product set.

\begin{corollary}
		Let $A\subseteq \F_p$ be  a set such that $|A| \le \d p/16$ and  $|AA| \le (2-\d)|A|$ for some $\d \in (0,1]$. 
	Suppose that $\sf(A) = |A|/3 + \psi$
	then for any $\eps>0$  one has 
	\begin{equation*}\label{f:Kneser_appl}
		\psi \gg_\eps \d^{1/2} |A| p^{-1/2 - \eps}   \,.
	\end{equation*} 
\label{c:Kneser_appl}
\end{corollary}
\begin{proof} 
	Kneser's theorem  \cite{Kneser}  implies that there is a subgroup $\G \subseteq \F_p^*$ with  $|\G| \le (2-\d) |A|$ such that $AA$ is covered 
	by at most $\frac{2}{\d}-1 := t$ translates of $\G$. Put $\G_* = \G \cup (-\G)$ then clearly $A\sbeq \bigcup_{x\in X}x\G_*$ for some set $X$ of size at most $t.$
	Using the Cauchy--Schwarz inequality, Theorem \ref{t:main_intr} and our assumption  $|A| \le \d p/16$, we have 
\[
	\frac{\d |A|^2}{8}
	\le 
	\frac{|A|^2}{4|X|}
	\le
	\frac{1}{|X|}
	\left( |A| - \frac{|A| |X||\G_*|}{p-1} \right)^2  
		\le 
			\sum_{\xi \in X} \left| |A \cap \xi \G_*| - \frac{|A||\G_*|}{p-1} \right|^2  \ll_\eps  \psi^2 p^{1+\eps} \,.
\]  
	This completes the proof.
$\hfill\Box$
\end{proof}

\section{Sum-free subsets in multiplicative subgroups}

In view of the results in the previous section   a natural problem of determining $\sf (\Gamma)$ for a multiplicative subgroup $\Gamma \sbeq \Fp^*$ arises. It is well--known that large multiplicative subgroups are pseudo--random sets, as they have small Fourier coefficients. Therefore, one can expect that $\sf (\Gamma)\le (1/3+o(1))|\Gamma|.$ Our next theorem will confirm this intuition.  The idea behind the proof is that if $\G$ contains a large sum-free set, then we show that there is a sum--free set of $\Fp$ roughly with the same density.  Our argument is based on  Fourier approximation method.

\begin{theorem}\label{t: s-f-G} Let $\G \subseteq \F_p^*$ be a multiplicative subgroup such that $|\Gamma|\gg \frac{(\log\log p)^2}{\sqrt{\log p}}p.$ Then 
$$\sf (\Gamma)\le (1/3+o(1))|\Gamma|.$$ \label{t:sf_G}
\end{theorem}
\begin{proof} 
Let us recall a simple bound on the Fourier coefficients of a multiplicative subgroup. For every $\xi_0\in \Fp^*,$ we have
$$|\G||\FF \G(\xi_0)|^2\le \sum_\xi|\FF \G(\xi)|^2= p|\G|\,,$$
hence $|\FF \G(\xi_0)|\le \sqrt{p},$ and it follows that
$$|m\FF\Gamma(\xi)-\FF\Fp(\xi)|\le m\max_{\xi\not=0}|\FF\Gamma(\xi)|\le m\sqrt{p}\,,$$
where $m:=p/|\Gamma|.$ Let $A$ be a sum--free subset of $\Gamma$ of the maximum size and put $|A|=\g |\G|.$ We show that there exists a sum--free subset in $\Fp$ roughly with the same density. For $\t>0$ define by
$$\Spec=\Spec_\t(A):=\big\{\xi: |\FF A(\xi)|\ge \t|A|\big \}$$
the $\t$--spectrum of $A$ and let
$$B=B(\Spec,\t)=\{b\in \Fp: \|b\xi/p\|\le \t \text{ for every } \xi\in\Spec\}\,.$$
 Furthermore, let us denote  $\b(x)=\frac1{|B|}B(x), ~a(x)=m\cdot A(x), ~g(x)=m\cdot \G(x)$ and define 
$$f=a*\b*\b\,.$$
Notice that $\FF f=\FF a\cdot \FF \b^2,$ hence  
\begin{equation}\label{fourier-approx}
|\FF a(\xi)-\FF f(\xi)|=|\FF a(\xi)||1-\FF \b(\xi)^2|\ll \t m|A|\,,
\end{equation}
for any $\xi\in \Fp.$
We will use the function $f$ to construct a subset of $\Fp$ with similar properties and therefore  we need to  estimate $\|f\|_1$  and  bound $f$ from above.  By Fourier inversion and by \eqref{fourier-approx}, we have
$$\big|\sum_t f(t)-m|A|\big|=\big|\sum_t f(t)-\sum_t a(t)\big|\ll \t m |A|\,,$$
and
\begin{eqnarray*}
f(t)&=& \frac1p\sum_\xi \FF f(\xi)e(-\xi t/p)=\frac1p\sum_\xi \FF a(\xi) \FF \b^2(\xi)e(-\xi t/p)\\
&\le&\frac1p \sum_\xi \FF g(\xi) \FF \b^2(\xi)e(-\xi t/p)\le \frac1p\sum_\xi \FF \Fp(\xi) \FF \b^2(\xi)e(-\xi t/p)+ 
\frac1p m \sqrt{p} \sum_\xi |\FF \b(\xi)|^2\\
&\le& 1+\frac{m\sqrt{p}}{|B|}=:1+\d\,.
\end{eqnarray*}
and we assume that $\g\le \t$ (a choice of parameters at the end of the proof will guarantee this inequality).

For  a function $w: \Fp\rightarrow {\mathbb R}$ put
$$T(w):=\sum_{x+y=z}w(x)w(y)w(z)\,.$$ 
We show by \eqref{fourier-approx}  that  $T(f)$ is small
\begin{eqnarray*}
T(f)&=&T(f)-T(a)=\frac1p\sum_\xi\FF f(\xi)|\FF f(\xi)|^2-\frac1p\sum_\xi\FF a(\xi)|\FF a(\xi)|^2\\
&\ll & \frac1p\sum_\xi\FF a(\xi)(|\FF f(\xi)|^2-|\FF a(\xi)|^2)+\frac1p\t m |A|\sum_\xi|\FF f(\xi)|^2\ll \t m^3|A|^2\,.
\end{eqnarray*}

Let 
$$h:=\frac1{1+\d}\cdot f$$ then we see that $h: \Fp\rightarrow [0,1],~\sum_t h(t)=m|A|+O(\t m|A|)$ and $T(h)\ll \t m^3|A|^2.$
By a corollary to Beck--Spencer theorem (see \cite{Green-bs}) there is a set $S\sbeq \Fp$  such that  
$$|S|=\lfloor \sum_t h(t)\rfloor=m|A|+O(\t m|A|)$$ and for every $\xi \in \Fp$
$$\big|\FF S(\xi)-\FF h(\xi)\big|\ll \sqrt{p}\,.$$ 
By the above property, we have 
$$\big |T(S)-T(a)\big|\le \big |T(S)-T(h)\big| +(1+\d)^{-3}\big |T(f)-T(a)\big|\ll \t m|S|^2\,.$$
Now it is sufficient to apply arithmetic removal lemma \cite{Green-removal} to find  set $S'\sbeq S$ such that 
$$|S'|\ge (1-\eps)|S|\ge (1-\eps+O(\t))m\,\sf(\Gamma)$$
 and $S'$ does not contain any solution to $x+y=z,$  where $\eps\rightarrow 0$ as $\t m\rightarrow 0.$ However, it immediately follows from the Cauchy--Davenport theorem that every sum--free set in $\Fp$ is of size at most $(p+1)/3$, so 
$$\sf(\Gamma)\le (1+O(\eps))\frac{|S'|}{m}\le (1/3+O(\eps+\t ))|\Gamma|\,.$$

To  finish the proof it is enough to find a choice of $\t$ such that $\t m\rightarrow 0$ and $\d\le \t.$
We can  assume that $|A|\ge \frac{|\G|}{4} \ge\frac{p}{4m}$ hence by Chang's spectral lemma \cite{chang} and by  a well--known lower bound for the size of a Bohr set, we have 
$$|B|\ge (\t/C\log m)^{C\t^{-2}\log m}p\,,$$
for some constant $C>0.$
We show that one can take $\t=c\frac{\log \log p}{\sqrt{\log p}}$ for a sufficiently small constant $c>0.$ 
By our assumption $m\ll \frac{\sqrt{\log p}}{(\log\log p)^2},$ so 
$\t m\rightarrow 0$ and
$$\d=\frac{m\sqrt{p}}{|B|}\le \frac{m}{ (\t/C\log m)^{C\t^{-2}\log m}\sqrt{p}}\le \t\,,$$
provided that $c$ is small enough.
This completes the proof.$\hfill\Box$
\end{proof}

\begin{remark} To keep the statement of Theorem \ref{t: s-f-G} simple  we pick  a suboptimal $m$; the optimal choice is $m=o(\t^{-1}).$ It is worth noting that Alon and Bourgain \cite{AB} constructed a sum-free subgroup of $\Fp^*$ of size $\gg p^{1/3}.$
\end{remark}

\section{Concluding remarks}

Bourgain showed \cite{Bourgain} that 
\begin{equation}\label{f:Wiener}
	\sf_2^{\Z} (A) \ge \frac{|A|}{3} + \frac{c \| A\|_{\Re W}}{\log |A|} \,,
\end{equation}
where $c>0$ is an absolute constant and $\| A\|_{\Re W} = \int_{0}^{1} |\sum_{a\in A} \cos (2\pi a x)|\, dx$.
One can verify that the same argument  works in the finite fields and the inequality \eqref{f:Wiener} holds for $\sf (A)$ with 
\begin{equation*}\label{f:RWiener_Fp}
	\| A\|_{\Re W} = p^{-1} \sum_x |\sum_{a\in A} \cos (2\pi a x/p)| \,.
\end{equation*}
Note that $\| A\|_{\Re W}$ can differ a lot from the usual Wiener norm $\| A\|_{W} = p^{-1} \sum_x |\FF{A} (x)|$, for example  if $p\equiv 3 \pmod 4$ then  $\| \mathcal{Q} \|_{\Re W} < 1,$ while  $\| \mathcal{Q} \|_{ W} \gg \sqrt{p}.$
However, for symmetric sets  $A=-A$ we have  $\| A\|_{\Re W}=\| A\|_{W}$ and
whence if $\sf(A) = |A|/3+\psi$ then
\[
	\left| |\xi \G \cap A| - \frac{|A||\G|}{p} \right| = \frac{1}{p} \sum_{x \neq 0} \FF{A}(x) \FF{\G} (x\xi) \le \| A\|_W \cdot \max_{x\neq 0} |\FF{\G} (x) |\,,
\]
so
\[
	\left| |\xi \G \cap A| - \frac{|A||\G|}{p} \right| \ll \psi \log |A| \cdot  \max_{x\neq 0} |\FF{\G} (x) | \,.
\]
Nevertheless,  the inequality above does not improve the  $L_2$--bound given in \eqref{f:disc_sf_G_intr}. 

\bigskip

\noindent {Faculty of Mathematics and Computer Science,\\ Adam Mickiewicz
University,\\ Umul\-towska 87, 61-614 Pozna\'n, Poland\\} {\tt
schoen@amu.edu.pl}

\bigskip

\noindent{I.D.~Shkredov\\
Steklov Mathematical Institute,\\
ul. Gubkina, 8, Moscow, Russia, 119991}
\\
and
\\
IITP RAS,  \\
Bolshoy Karetny per. 19, Moscow, Russia, 127994\\
and 
\\
MIPT, \\ 
Institutskii per. 9, Dolgoprudnii, Russia, 141701\\
{\tt ilya.shkredov@gmail.com}

\end{document}